\documentclass[10pt, reqno]{amsart}
\usepackage{amssymb}
\usepackage{amsmath,amscd}
\usepackage[initials,nobysame,alphabetic]{amsrefs}
\usepackage{amsfonts}
\usepackage{mathrsfs}
\usepackage{mathpazo}
\usepackage[arrow,matrix,curve,cmtip,ps]{xy}
\usepackage{amsthm}
\usepackage{float}
\usepackage{hyperref}
\usepackage{graphics}
\usepackage{graphicx}
\usepackage{verbatim}
\usepackage{multirow}
\usepackage{tikz}
\usepackage{enumerate}
\usepackage{eufrak}
\usepackage[normalem]{ulem}
\allowdisplaybreaks

\calclayout
\newtheorem{theorem}{Theorem}[section]
\newtheorem{lemma}[theorem]{Lemma}
\newtheorem{proposition}[theorem]{Proposition}

\theoremstyle{remark}
\newtheorem{remark}[theorem]{Remark}
\newtheorem{definition}[theorem]{Definition}

\numberwithin{equation}{section}

\newcommand{\Om}{\Omega}

\newcommand{\E}{\mathbb{E}}
\newcommand{\R}{\mathbb{R}}
\newcommand{\Ff}{\mathbb{F}}

\newcommand{\U}{\mathcal{U}}

\newcommand{\V}{\mathcal{V}}
\newcommand{\C}{\mathcal{C}}
\newcommand{\F}{\mathcal{F}}
\newcommand{\cs}{\mathcal{S}}
\newcommand{\ch}{\mathcal{H}}
\newcommand{\Prob}{\mathbb{\Prob}}

\newcommand{\mytilde}{\raise.17ex\hbox{$\scriptstyle\mathtt{\sim}$}}

\def \Om {\Omega}

\def \rw {\rightarrow}

\newcommand{\e}{\varepsilon}
\begin{document}
\title[Optimal control and zero-sum games of mean-field type]{Optimal control and zero-sum stochastic differential game problems of mean-field type}

\author{Boualem Djehiche and Said Hamad\`ene}

\address{Department of Mathematics \\ KTH Royal Institute of Technology \\ 100 44, Stockholm \\ Sweden}
\email{boualem@math.kth.se}
\address{Universit\'e du Maine, LMM \\ Avenue Olivier Messiaen \\ 72085 Le Mans, Cedex 9, France}
\email{ hamadene@univ-lemans.fr}


\date{First version November 17, 2016. This version \today}

\subjclass[2010]{60H10, 60H07, 49N90}

\keywords{mean-field, nonlinear diffusion process, backward SDEs,
optimal control, zero-sum game,  saddle-point}

\begin{abstract}
We establish existence of nearly-optimal controls,   conditions for
existence of an optimal control and a saddle-point  for respectively
a control problem and zero-sum differential game associated with
payoff functionals of mean-field type, under  dynamics driven by
weak solutions of stochastic differential equations of mean-field
type.
\end{abstract}

\maketitle

\tableofcontents

\section{Introduction}
In this work we investigate existence of an optimal control and a
saddle-point for a zero-sum game associated with a payoff functional
of mean-field type, under a dynamics driven by the weak solution of
a  stochastic differential equation (SDE) also of mean-field type.
The obtained results extend in a natural way those obtained in
\cite{Ham-Lepl95} for standard  payoffs associated with standard
diffusion processes.

Given a control process $u:=(u_t)_{t\leq T}$ with values in some
compact metric space $U$, the controlled SDE of mean-field type we
consider in this paper is of the following functional form:
\begin{equation}\label{solfaible}
  dx_t=f(t,x_.,P^u\circ x_t^{-1},u_t)dt+\sigma(t,x_.)dW^{P^u}_t,\quad x_0=\xi\in \R^d,
\end{equation}
i.e. the $f$ and $\sigma$ depend on the whole path $x_.$ and
$P^u\circ x_t^{-1}$ (this feature can be improved substantially, see
Remark \ref{remimprove}), the marginal probability distribution of
$x_t$ under the probability measure $P^u$, and where $W^{P^u}$ is a
standard Brownian motion under $P^u$.  The payoff functional
$J(u),\,\, u\in\U,$ associated with the controlled SDE is of the
form
\begin{equation*}
J(u):=E^u\left[\int_0^T h(t,x_.,P^u\circ x_t^{-1},u_t)dt+
g(x_T,P^u\circ x_T^{-1})\right],
\end{equation*}
where $E^u$ denotes the expectation w.r.t. $P^u$. \\
As an example, the functions $f$, $g$ and $h$ can have the following
forms
$$f(t,x_.,E^u[\varphi_1(x_t)],u), g(x,E^u[\varphi_2(x_T)]) \mbox{ and }
h(t,x_.,E^u[\varphi_3(x_t)],u)$$ where $\varphi_i$, $i=1,2,3$, are
bounded Borel-measurable functions.\\Taking $h=0$  and
$g(x,y)=\varphi_2(x)^2-y^2$,  the cost functional reduces to the
variance, $J(u)=E^u[\varphi_2(x_T)^2])-\left(E^u[\varphi_2
(x_T)]\right)^2=Var_{P^u}[\varphi_2(x_T)]$.\\

While controlling a strong solution of an SDE means controlling the
process $x^u$ defined on a given probability space $(\Omega, \F,\Ff,
P)$ on which a Brownian motion $W$ is defined exists and $\Ff$ is
its natural filtration, controlling a weak solution of an SDE boils
down to controlling the Girsanov density process $L^u:=dP^u/dP$ of
$P^u$ w.r.t. a reference probability measure $P$  on $\Omega$ such
that $(\Omega,P)$ carries a Brownian motion $W$ and such that the
coordinates process $x_t$ is the unique solution of the following
stochastic differential equation:
\begin{equation*}\label{free-SDE}
dx_t=\sigma(t,x_.)dW_t,\quad x_0=\xi.
\end{equation*}
Integrating by parts,  the payoff functional can be expressed in
terms of $L^u$ as follows
$$
J(u)=E\left[\int_0^T L^u_th(t,x_.,P^u\circ x_t^{-1},u_t)dt+
L^u_Tg(x_T,P^u\circ x_T^{-1})\right],
$$
where $E$ denotes the expectation w.r.t. $P$. For this reason, we do
not include a control parameter in the diffusion term $\sigma$.

 In the first part of this paper we establish conditions for existence of an optimal control associated with $J(u)$:  Find  a stochastic process $u^*$ with values in $U$ such that
\begin{equation*}\label{opt-J}
J(u^*)=\min_{u\in \U}J(u).
\end{equation*}

The recent paper by Carmona and Lacker \cite{Car-Lacker} discusses a
similar problem but in the so-called mean-field game setting (where
they further consider the marginal laws of the control process,
i.e., $P^u\circ u_t^{-1}$) which has the following structure (cf.
\cite{Car-Lacker}):
\begin{enumerate}
\item  Fix a probability measure $\mu$ on the path space and a flow $\nu:   t \mapsto \nu_t$ of measures
on the control space;
\item  Standard optimization: With $\mu$ and $\nu$ frozen, solve the standard optimal control problem:
\begin{equation}\label{mf-game}
\left\{\begin{array}{lll}
\inf_u E^u\left[\int_0^T h(t,x_.,\mu,\nu, u_t)dt+ g(x_T,\mu)\right], \\
dx_t=f(t,x_.,\mu, u_t)dt+\sigma(t, x_.)dW^{P^u}_t,\quad x_0=\xi\in
\R^d,
\end{array}
\right.
\end{equation}
i.e. find an optimal control $u$, inject it into the dynamics of
(\ref{mf-game}), and find the law $\Phi_x(\mu,\nu)$ of the optimally
controlled state process and the flow $\Phi_u(\mu,\nu)$ of marginal
laws of the optimal control process;
\item Matching: Find a fixed point $\mu=\Phi_x(\mu,\nu),\,\, \nu=\Phi_u(\mu,\nu)$.
\end{enumerate}
To perform the matching step (3), the authors of \cite{Car-Lacker}
are led to impose more or less stringent assumptions which in turn
narrow the scope of the applicability of their framework. This is
mainly due to the fact that the functional which is supposed to
provide the optimal control is rather irregular. Overall, to show
existence of a fixed point is not an easy task and cannot work in
broader frameworks. For further details about the mean-field games
approach see also \cite{Bensoussan-MF} and the references therein.

In this paper we use another approach which in a way  addresses the
full control problem where the marginal law changes with the control
process and is {\it not frozen} as in the mean-field game approach.
Our strategy goes as follows: By a fixed point argument we first
show that for any admissible control $u$ there exists a unique
probability $P^u$ under which the SDE
$$
  dx_t=f(t,x_.,P^u\circ x_t^{-1},u_t)dt+\sigma(t,x_.)dW^{P^u}_t,\quad x_0=\xi\in \R^d,
$$
has a weak solution, where $W^{P^u}$ is a Brownian motion under
${P^u}$. Moreover, the mapping which to $u$ associates $P^u$ is
continuous.  Therefore, the mean-field terms which appear in the
drift of the above equation and in the payoff functional $J(u)$ are
treated as continuous functions of $u$. Using this point of view,
which avoids the irregularity issues encountered in
\cite{Car-Lacker}, we suggest conditions for  existence of an
optimal control using backward stochastic differential equations
(BSDEs)  in a similar fashion the standard control problems, i.e.\@
without mean-field terms. Indeed, if $(Y^u,Z^u)$ is the solution of
the BSDEs associated with the driver (Hamiltonian)
$H(t,x_.,z,u):=h(t,x_.,P^{u}\circ x_t^{-1},u_t)+z\cdot
\sigma^{-1}(t,x_.)f(t,x_.,P^u\circ x_t^{-1},u_t)$ and the terminal
value $g(x_T,P^{u}\circ x^{-1}_T)$, we have $Y^u_0=J(u)$. Moreover,
the unique solution $(Y^*,Z^*)$ of the BSDE associated with
$$
H^*(t,x_.,z):=\underset{u\in \U}{\mathrm{ess}\inf\,}
H(t,x_.,z,u),\,\, g^*(x_.):=\underset{u\in
\U}{\mathrm{ess}\inf\,}g(x_T,P^{u}\circ x^{-1}_T)
$$
satisfies, under appropriate assumptions, $Y^*(t)=\underset{u\in
\U}{\mathrm{ess}\inf\,}Y^u(t)$. The use of the essential infimum
over the whole set of admissible controls $\U$ instead of the
infimum of the Hamiltonian $H$ over the set $U$ of actions (as is
the case for the standard control problem, as discussed  e.g. in
\cite{Ham-Lepl95})  is simply due to the fact that the mean-field
coupling $P^{u}\circ x_t^{-1}$ involves the whole path of the
control $u$ over $[0,t]$ and not only on $u_t$. This nonlocal
feature of the dependence of $H$ on the control does not  seem
covered by the powerful Benes' type 'progressively' measurable
selection, frequently used in standard control problems. Thus, if
there exists $u^*\in\U$ such that $H^*(t,x_.,z)=H(t,x_.,z,u^*)$ and
$g^*(x_.)=g(x_T,P^{u^*}\circ x^{-1}_T)$, then $u^*$ is an optimal
control for $J(u)$. We don't know of any suitable measurable
selection theorem that would guarantee existence of $u^*$.

The zero-sum game we consider is between two players with controls
$u$ and $v$ valued in some compact metric spaces $U$ and $V$,
respectively. The dynamics and the payoff function associated with
the game are both of mean-field type and are given by
\begin{equation}\label{sde-game}
  dx_t=f(t,x_.,P^{u,v}\circ x_t^{-1},u_t,v_t)dt+\sigma(t,x_.)dW^{P^{u,v}}_t,\quad x_0=\xi\in \R^d,
\end{equation}
and
\begin{equation*}
J(u,v):=E^{u,v}\left[\int_0^T h(t,x_.,P^{u,v}\circ
x_t^{-1},u_t,v_t)dt+ g(x_T,P^{u,v}\circ x^{-1}_T)\right],
\end{equation*}
where $P^{u,v}\circ x_t^{-1}$ is the marginal probability
distribution of $x_t$ under the probability measure $P^{u,v}$,
$W^{P^{u,v}}$ is a standard Brownian motion under $P^{u,v}$ and
$E^{u,v}$ denotes the expectation w.r.t. $P^{u,v}$.

In the zero-sum game, the first player (with control $u$) wants to
minimize the payoff $J(u,v)$ while  the second player (with control
$v$) wants to maximize it.  The zero-sum game boils down to
investigating the existence of a saddle point for the game i.e.  to
show existence of a pair $(u^*, v^*)$ of strategies such that
\begin{equation*}
J(u^*, v) \le J(u^*, v^*)\le J(u,v^*),
\end{equation*}
for each $(u, v)$ with values in $U\times V$. By using the same
approach as in the control framework, we show that the game has a
saddle-point. The recent paper by Li and Min \cite{limin} deals with
the same zero-sum game for weak solutions of SDEs of the form
(\ref{solfaible}), where they apply a similar 'matching argument'
approach as \cite{Car-Lacker}. However, due to the irregularity of
the functional which provides the fixed point, they could only show
existence of a so-called generalized saddle-point i.e.  of a pair of
strategies $(u^*, v^*)$ which satisfies (see, for instance, Theorem
5.6 in \cite{limin})
\begin{equation*}
J(u^*, v)-C\psi(v,v^*) \le J(u^*, v^*)\le J(u,v^*)+C\psi(u,u^*),
\end{equation*}
where $\psi(u,\bar u):=(E[\int_0^T d^2(u_s.\bar u_s)ds])^{1/4}$ and
$C$ is a positive constant depending only on $f$ and $h$.

Instead of the Wasserstein metric which is by now standard in the
literature dealing with mean-field models, because it is designed to
guarantee weak convergence of probability measures and convergence
of finite moments, in this paper we have chosen to use the total
variation as a metric between two probability measures, although it
does not guarantee existence of finite moments, simply due to its
relationship to the Hellinger distance thanks to the celebrated
Csisz{\'a}r-Kullback-Pinsker inequality (see the bound (4.22),
Theorem V.4.21 in \cite{Jacod-shir}) which gives a simple and direct
proof of existence of a unique probability $P^u$  (resp. $P^{u,v}$)
under which the SDE \eqref{solfaible} (resp. \eqref{sde-game}) has a
weak solution.

The paper is organized as follows. In Section \ref{sde-mf}, we
account for existence and uniqueness of the weak solution of the SDE
of mean-field type. In Section \ref{control}, we  provide conditions
for existence of an optimal control and prove existence of
nearly-optimal controls. Finally, in Section \ref{zero-sum}, we
investigate existence of a saddle point for a two-persons zero-sum
game.


\section{Preliminaries}\label{Preliminaries}


Let $\Om:=\mathcal{C}([0,T]; \R^d)$ be the space of $\R^d$-valued continuous functions on $[0,T]$ endowed with the metric of uniform convergence on $[0,T]$; $|w|_t:=\sup_{0\le s\le t}|w_s|$, for $0\le t\le T$. Denote by $\F$ the Borel $\sigma$-field over $\Om$. Given $t\in [0,T]$ and $\omega\in\Om$, let $x(t,\omega)$ be the position in $\R^d$ of $\omega$ at time $t$. Denote by $\F^0_t:=\sigma(x_s,\,\, s\le t),\, 0\le t\le T,$ the filtration generated by $x$. Below,  $C$ denotes a generic positive constant which may change from line to line.\\

Let $\sigma$ be a function from $[0,T]\times\Om$ into $\R^{d\times
d}$ such that
\begin{itemize}
\item[(A1)]  $\sigma$ is $\F^0_t$-progressively measurable ;
\item[(A2)] There exists a constant $C>0$ such that
\begin{itemize}
\item[(a)] For every $t\in[0,T]$ and  $w, \bar w \in \Om$, $
|\sigma(t,w)-\sigma(t,\bar w)|\le C|w-\bar w|_t.$

\item[(b)] $\sigma$ is invertible and its inverse $\sigma^{-1}$  satisfies $|\sigma^{-1}(t,w)|\le C(1+|w|_t^{\alpha}),$ for some constant $\alpha \ge 0$.

\item[(c)] For every $t\in[0,T]$ and  $w\in \Om$, $|\sigma(t,w)|\le C(1+|w|_t).$
\end{itemize}
\end{itemize}

Let $P$ be a probability measure on $\Omega$ such that $(\Omega,P)$
carries a Brownian motion $(W_t)_{0\le t\leq T}$ and such that the
coordinates process $(x_t)_{0\le t\leq T}$ is the unique solution of
the following stochastic differential equation:
\begin{equation}\label{free-SDE}
dx_t=\sigma(t,x_.)dW_t,\quad x_0=\xi \in \R^d.
\end{equation}
Such a triplet $(P,W,x)$ exists due to Proposition 4.6 in
(\cite{Karatzas-Shreve}, p.315) since $\sigma$ satisfies (A2).
Moreover, for every $p\ge 1$,
 \begin{equation}\label{x-estim}
 E[|x|_T^{p}]\le C_p,
 \end{equation}
where $C_p$ depends only on $p, T$, the initial value $\xi$ and the linear growth constant of $\sigma$ (see \cite{Karatzas-Shreve}, p. 306). Again, since $\sigma$ satisfies (A2), $\F^0_t$ is the same as $\sigma\{W_s, s\leq t\}$ for any $t\leq T$. \\
 We denote by $\Ff:=(\F_t)_{0\le t\le T}$ the completion of $(\F^0_t)_{t\le T}$ with the $P$-null sets of $\Omega$.

\medskip
Let $\mathcal{P}(\R^d)$ denote the set of probability measures on
$\R^d$ and $\mathcal{P}_2(\R^d)$ the subset of measures with finite
second moment. For $\mu,\nu\in\mathcal{P}(\R^d)$, the total
variation distance is defined by the formula
\begin{equation}\label{TV-B}
d(\mu,\nu)=2\sup_{B\in\mathcal{B}(\R^d)}|\mu(B)-\nu(B)|.
\end{equation}
Furthermore, let $\mathcal{P}(\Omega)$ be the space of probability
measures $P$ on $\Om$ and  $\mathcal{P}_p(\Omega), \,p\ge 1,$ be the
subspace of probability measures such that $$
\|P\|_p^p:=\int_{\Om}|w|^p_TP(dw)=E[|x|_T^p]<+ \infty,
$$
where $|x|_t:=\sup_{0\le s\le t}|x_s|$, $0\le t\le T$. \\
Define on $\F$ the total variation metric
\begin{equation}\label{TV-F}
d(P,Q):=2\sup_{A\in\F}|P(A)-Q(A)|.
\end{equation}
Similarly, on the filtration $\Ff$, we define the total variation
metric between two probability measures $P$ and $Q$ as
\begin{equation}\label{TV-filt}
D_t(P,Q):=2\sup_{A\in\F_t}|P(A)-Q(A)|,\quad 0\le t\le T.
\end{equation}
It satisfies
\begin{equation}
D_s(P,Q)\le D_t(P,Q),\quad 0\le s\le t.
\end{equation}
For $P, Q\in \mathcal{P}(\Om)$ with time marginals $P_t:=P\circ
x_t^{-1}$ and $Q_t:=Q\circ x_t^{-1}$, the total variation distance
between $P_t$ and $Q_t$ satisfies
\begin{equation}\label{margine}
d(P_t,Q_t)\le D_t(P,Q),\quad 0\le t\le T.
\end{equation}
Indeed, we have
\begin{equation*}\begin{array}{lll}
d(P_t,Q_t):=2\sup_{B\in\mathcal{B}(\R^d)}|P_t(B)-Q_t(B)|\\\qquad\quad\quad\;\;
=2\sup_{B\in\mathcal{B}(\R^d)}|P(x_t^{-1}(B))-Q(x_t^{-1}(B))|\\\qquad\quad\quad\;\;
\le 2\sup_{A\in\mathcal{F}_t}|P(A)-Q(A)|=D_t(P,Q).
\end{array}
\end{equation*}

\medskip
Endowed with the total variation metric $D_T$, $\mathcal{P}(\Om)$ is
a complete metric space. Moreover, $D_T$ carries out the usual
topology of weak convergence.

\section{Diffusion process of mean-field type}\label{sde-mf}

\noindent Hereafter, a process $\theta$ from $[0,T]\times \Om$ into
a measurable space is said to be progressively measurable if it is
progressively measurable w.r.t. $\Ff$. Let ${\cs}^2_T$  be the set
of $\Ff$-progressively measurable continuous processes
$(\zeta_t)_{t\leq T}$ such that $E[\sup_{t\leq
T}|\zeta_t|^2]<\infty$ and finally let ${\ch}^2_T$  be the set of
$\Ff$-progressively measurable processes $(\theta_t)_{t\leq T}$ such
that $E[\int_0^T|\theta_s|^2ds]<\infty$.
\bigskip

\noindent Let $b$ be a measurable function from
$[0,T]\times\Om\times \mathcal{P}(\R^d)$ into $\R^d$ such that
 \begin{itemize}
 \item[(A3)] For every $Q\in \mathcal{P}(\Om)$, the process $((b(t, x_.,Q\circ x_t^{-1}))_{t\leq T}$ is progressively measurable.

 \item[(A4)] For every $t\in[0,T]$, $w\in \Om$ and  $\mu, \nu \in\mathcal{P}(\R^d)$,
  $$
  |b(t,w,\mu)-b(t,w,\nu)|\le Cd(\mu,\nu).
  $$
 \item[(A5)] For every $t\in[0,T]$, $w\in \Om$ and $\mu\in\mathcal{P}(\R^d)$,
 $$
 |b(t,w,\mu)|\le C(1+|w|_t).
 $$
 \end{itemize}

\noindent  Next, for $Q\in \mathcal{P}(\Omega)$, let $P^Q$ be the
measure on $(\Om,\F)$ defined by
 \begin{equation}\label{PQ}
 dP^Q:=L_T^Q dP
 \end{equation}
with
  \begin{equation}\label{PQ-density}
 L_t^Q:=\mathcal{E}_t\left(\int_0^{\cdot} \sigma^{-1}(s,x_{\cdot})b(s,x_{\cdot},Q\circ x_s^{-1})dW_s\right),\quad 0\le t\le T,
 \end{equation}
 where, for any $(\Ff,P)$-continuous local martingale $M=(M_t)_{0\le t\le T}$, $\mathcal{E}(M)$ denotes the Doleans exponential $\mathcal{E}(M):=(\exp{M_t-\frac{1}{2}\langle M\rangle_t})_{{0\le t\le T}}$. Thanks to assumptions (A2) and (A5), $P^Q$ is a probability measure on  $(\Omega,\F)$. A proof of this fact follows the same lines of the proof of Proposition A.1 in \cite{Elkar-Ham}. Hence, in view of Girsanov's theorem, the process $(W^Q_t,\,\, 0\le t\le T)$ defined by
 $$
 W_t^Q:=W_t-\int_0^t \sigma^{-1}(s,x_.)b(s,x_.,Q\circ x_s^{-1})ds, \quad 0\le t\le T,
 $$
 is an $(\Ff, P^Q)$-Brownian motion. Furthermore, under $P^{Q}$,
 \begin{equation}\label{SDE-2}
  dx_t=b(t,x_.,Q\circ x_t^{-1})dt+\sigma(t,x_.)dW^Q_t,\quad x_0=\xi\in \R^d.
 \end{equation}
 Furthermore, in view of (A2) and (A5), the H\"older and Burkholder-Davis-Gundy inequalities yield, for every $p\ge 1$,
$$
\|P^Q\|_p^p=E_{P^Q}\left[|x|_T^p\right]\le
C_p\left(1+E_{P^Q}\left[\int_0^T |x|_t^p dt\right]\right).
$$
where the constant $C_p$ depends only on $p, T,\xi$ and the linear
growth constants of $b$ and $\sigma$. By Gronwall's inequality, we
obtain
\begin{equation}\label{P(Q)-estim}
E_{P^Q}[|x|^p_T]\le C_p<+\infty.
\end{equation}

Next, we will show that there is $\bar Q$ such that $P^{\bar
Q}={\bar Q}$, i.e., $\bar Q$ is a fixed point. Moreover, $\bar Q$
has a finite  moment of any order $p\ge 1$.

\begin{theorem}\label{FP} The map
\begin{equation*}\begin{array}{lll}
\Phi: \mathcal{P}(\Omega)\longrightarrow \mathcal{P}(\Omega) \\
\qquad\quad Q  \mapsto \Phi(Q):=P^Q;\quad   dP^Q:=L_T^Q dP
\end{array}
\end{equation*}
admits a unique fixed point.

Moreover, for every $p\ge 1$, the fixed point, denoted $\bar Q$,
belongs to $\mathcal{P}_p(\Omega)$, i.e.
\begin{equation}\label{Q-bar-estim}
E_{\bar Q}[|x|^p_T]\le C_p<+\infty,
\end{equation}
where the constant $C_p$ depends only on $p, T,\xi$ and the linear
growth constants of $b$ and $\sigma$.
\end{theorem}
\begin{proof} We show the contraction property of the map $ \Phi$ in the complete metric space $\mathcal{P}(\Omega)$, endowed with the  total variation distance $D_T$. To this end, given $Q,\widehat{Q}\in\mathcal{P}(\Om)$, we use an estimate of the total variation distance $D_T(\Phi(Q),\Phi(\widehat{Q}))$ in terms of a version of the Hellinger process associated with the coordinate process $x$ under the probability measures $\Phi(Q)$ and $\Phi(\widehat{Q})$, respectively. Indeed, since by (\ref{SDE-2}),
\begin{equation*}\left\{\begin{array}{lll}
 \text{under}\,\, \Phi(Q),\;\; dx_t=b(t,x_.,Q_t)dt+\sigma(t,x_.)dW^Q_t,\quad x_0=\xi\in \R^d,\\ \\
 \text{under}\,\, \Phi(\widehat{Q}),\;\; dx_t=b(t,x_.,\widehat{Q}_t)dt+\sigma(t,x_.)dW^{\widehat{Q}}_t,\quad x_0=\xi\in \R^d,\\
 \end{array}
 \right.
\end{equation*}
in view of Theorem IV.1.33 in \cite{Jacod-shir}, a version of the
associated Hellinger process is
\begin{equation}\label{hellinger}
\Gamma_T:=\frac{1}{8}\int_0^T \Delta
b_t(Q,\widehat{Q})^{\dagger}a^{-1}_t\Delta b_t(Q,\widehat{Q})dt,
\end{equation}
where
\begin{equation*}\label{b-notation}
\Delta b_t(Q,\widehat{Q}):=b(t,x_.,Q_t)-b(t,x_.,\widehat{Q}_t)
\end{equation*}
and $a_t:=(\sigma\sigma^{\dagger})(t,x_.)$ and $M^{\dagger}$ denotes
the transpose of the matrix $M$. We may use the estimate (4.22) of
Theorem V.4.21 in \cite{Jacod-shir}, to obtain
\begin{equation}\label{estimate}
D_T(\Phi(Q),\Phi(\widehat{Q}))\le
8\sqrt{E_{\Phi(Q)}\left[\Gamma_T\right]}.
\end{equation}
By (A2), (A4) and \eqref{P(Q)-estim}, we have
\begin{equation*}
E_{\Phi(Q)}\left[\Delta b_t(Q,\widehat{Q})^{\dagger}a^{-1}_t\Delta
b_t(Q,\widehat{Q})\right]\le C d^2(Q_t,\widehat{Q}_t)\le
CD^2_t(Q,\widehat{Q}),
\end{equation*}
which together with (\ref{estimate}) yield
\begin{equation}
D^2_T(\Phi(Q),\Phi(\widehat{Q}))\le C\int_0^TD^2_t(Q,\widehat{Q})dt.
\end{equation}
Iterating this inequality, we obtain, for every $N>0$,
$$
D^2_T(\Phi^N(Q),\Phi^N(\widehat{Q}))\le
C^N\int_0^T\frac{(T-t)^{N-1}}{(N-1)!}D^2_t(Q,\widehat{Q})dt\le
\frac{C^NT^N}{N!}D^2_T(Q,\widehat{Q}),
$$
where $\Phi^N$ denotes the $N$-fold composition of the map $\Phi$. Hence, for $N$ large enough, $\Phi^N$ is a contraction which entails that $\Phi$ admits a unique fixed point.  \\

Let $\bar Q$ be such a fixed point for the map $\Phi$. Thus, under
$\bar Q$,
\begin{equation*}
 dx_t=b(t,x_.,\bar Q_t)dt+\sigma(t,x_.)dW^{\bar Q},\quad x_0=\xi\in \R^d,
\end{equation*}
where $\bar Q_t:=\bar Q\circ x_t^{-1}$. In view of assumptions (A2)
and (A5), the H\"older and Burkholder-Davis-Gundy inequalities yield
$$
\|\bar Q\|_p^p=E_{\bar Q}\left[|x|_T^p\right]\le C_p\left(1+E_{\bar
Q}\left[\int_0^T |x|_t^p dt\right]\right).
$$
By Gronwall's inequality, we obtain \eqref{Q-bar-estim} i.e.
\begin{equation*}
E_{\bar Q}[|x|^p_T]\le C_p<+\infty.
\end{equation*}
\end{proof}
\begin{remark} \label{remimprove} The dependence of the drift $b$ with respect to the law of $x_t$ under $Q$, i.e.,  $Q \circ x_t^{-1}$ can be relaxed substantially since we can replace this latter by $Q\circ \phi(t,x)^{-1}$ where
$\phi(t,x)$ is an adapted process. For example one can choose $\phi(t,x)=\sup_{0\le s\leq t}x_s$. The main point is the inequality (\ref{margine}) which still hold with a general adapted process $\phi(t,x)$. \\
\qed
 \end{remark}

\section{Optimal control of the diffusion process of mean-field type}\label{control}

Let $(U, \delta)$ be a compact metric space with its Borel field
$\mathcal{B}(U)$ and $\U$
 the set of $\Ff$-progressively measurable processes $u=(u_t)_{t\leq T}$ with values in $U$.
 We call $\U$ the set of admissible controls.\\

Next let $f$ and $h$ be two measurable functions from $[0,T]\times\Om\times\mathcal{P}(\R^d)\times U$ into $\R^d$ and $\R$, respectively, and $g$ be a measurable functions from $\R^d\times\mathcal{P}(\R^d)$ into $\R$ such that \\

\begin{itemize}
 \item[(B1)] For any $u\in\U$ and $Q\in \mathcal{P}(\Om)$, the processes $(f(t, x_.,Q\circ x_t^{-1},u_t))_t$ and $(h(t, x_.,Q\circ x_t^{-1},u_t))_t$ are progressively measurable. Moreover, $g(x_T,Q\circ x_T^{-1})$ is $\mathcal{F}_T$-measurable.
 \item[(B2)] For every $t\in[0,T]$, $w\in\Om$, $u,v\in U$ and $\mu, \nu \in\mathcal{P}(\R^d)$,
  $$
  |\phi(t,w,\mu, u)-\phi(t,w,\nu,v)|\le C(d(\mu,\nu)+\delta(u,v)).
  $$
  for $\phi\in\{f,h,g\}$.
   \item[(B3)] For every $t\in[0,T]$, $w\in\Om$, $\mu\in\mathcal{P}(\R^d)$ and  $u\in U$,
 $$
 |f(t,w,\mu,u)|\le C(1+|w|_t).
 $$
 \item[(B4)] $h$ and $g$ are uniformly bounded
  \end{itemize}

\medskip
For $u\in\U$, let $P^u$ be the probability measure on $(\Om,\F)$
which is a fixed point of $\Phi^u$ defined in the same was as in
Theorem \eqref{FP} except that the drift term $b(\cdot)$ depends
moreover on $u$ but this does not rise a major issue. Thus we have
 \begin{equation}\label{P-u}
 dP^u:=L_T^u dP,
 \end{equation}
where
\begin{equation}\label{P-u-density}
 L_t^u:=\mathcal{E}_t\left(\int_0^{\cdot} \sigma^{-1}(s,x_.)f(s,x_.,P^u\circ x_s^{-1},u_s)dW_s\right),\quad 0\le t\le T.
 \end{equation}
By Girsanov's theorem, the process $(W^u_t,\,\, 0\le t\le T)$
defined by
 $$
 W_t^u:=W_t-\int_0^t \sigma^{-1}(s,x_.)f(s,x_.,P^u\circ x_s^{-1},u_s)ds, \quad 0\le t\le T,
 $$
 is an $(\Ff, P^u)$-Brownian motion. Moreover, under $P^u$,
 \begin{equation}\label{SDE-u}
 dx_t=f(t,x,P^u\circ x_t^{-1},u_t)dt+\sigma(t,x)dW^u_t,\quad x_0=\xi\in \R^d.
 \end{equation}
Let $E^u$  denote the expectation w.r.t. $P^u$. In view of
\eqref{Q-bar-estim}, we have, for every $u\in\U$,
\begin{equation}\label{u-x-estim}
\|P^u\|_2^2=E^u[|x|^2_T]\le C<+\infty,
\end{equation}
where the constant $C$ depends only on $T,\xi$ and the linear growth
constants of $f$ and $\sigma$.

We also have the following estimate of the total variation between
$P^u$ and $P^v$.
\begin{lemma}\label{TV-u}
For every $u,v\in \U$, it holds that
\begin{equation}\label{TV-uv-1}
D_T^2(P^u,P^v)\le C E^u[\int_0^T\delta^2(u_t,v_t)dt].
\end{equation}
In particular, the function $u\mapsto P^u$ from $U$ into
$\mathcal{P}_2(\Om)$ is Lipschitz continuous: for every $u,v\in U$,
\begin{equation}\label{TV-uv-2}
D_T(P^u,P^v)\le C\delta(u,v).
\end{equation}
Moreover,
\begin{equation}\label{K}
K_T:=\sup_{u\in U}\|P^u\|_2\le C<\infty,
\end{equation}
for some constant $C>0$ that depends only on $T,\xi$ and the linear
growth constants of $f$ and $\sigma$.
\end{lemma}
\begin{proof} Using a similar estimate as (\ref{estimate}), we have
\begin{equation}\label{estimate-uv}
D_T(P^u,P^v)\le 8\sqrt{E^u\left[\tilde\Gamma^{u,v}_T\right]},
\end{equation}
where $\tilde\Gamma$ is the following version of the Hellinger
process associated with $P^u$ and $P^v$:
$$
\tilde\Gamma_T:=\frac{1}{8}\int_0^T \Delta
f_t(u,v)^{\dagger}a^{-1}_t\Delta f_t(u,v)dt,
$$
where
$$
\Delta f_t(u,v):=f(t,x_.,P^u\circ x_t^{-1},u_t)-f(t,x_.,P^v\circ
x_t^{-1},v_t).
$$
Using (A2) and (B2), we obtain
$$\begin{array}{lll}
\Delta f_t(u,v)^{\dagger}a^{-1}_t\Delta f_t(u,v)\le C(d^2(P^u\circ
x_t^{-1},P^v\circ x_t^{-1})+\delta^2(u_t,v_t))\\
\qquad\qquad\qquad\qquad\qquad \le
C(D^2_t(P^u,P^v)+\delta^2(u_t,v_t)).
\end{array}
$$
Hence, in view of (\ref{estimate-uv}), Gronwall's inequality yields
\begin{equation*}
D^2_T(P^u,P^v)\le CE^u\left[\int_0^T \delta^2(u_t,v_t)dt\right].
\end{equation*}
Inequality (\ref{TV-uv-2}) follows from (\ref{TV-uv-1}) by letting  $u_t\equiv u\in U$ and $v_t\equiv v\in U$. \\
It remains to show (\ref{K}). But, this follows from
(\ref{u-x-estim}) and the continuity of the function $u\mapsto P^u$
from the compact set $U$ into $\mathcal{P}_2(\Om)$.
\end{proof}

\medskip
 The cost functional $J(u),\,\, u\in\U$, associated with the controlled SDE (\ref{SDE-u}) is
\begin{equation}\label{J-u}
J(u):=E^u\left[\int_0^T h(t,x_.,P^u\circ x_t^{-1},u_t)dt+
g(x_T,P^u\circ x_T^{-1})\right],
\end{equation}
where $h$ and $g$ satisfy (B1)-(B4) above.  \\

\noindent Any $u^*\in\U$ satisfying
\begin{equation}\label{opt-J}
J(u^*)=\min_{u\in\U}J(u)
\end{equation}
is called optimal control. The corresponding optimal dynamics is
given by the probability measure $\widehat P$ on $(\Om,\F)$ defined
by
\begin{equation}\label{opt-P}
dP^*=\mathcal{E}\left(\int_0^{\cdot} \sigma^{-1}(s,x_.)f(s,x_.,
P^*\circ x_s^{-1},u^*_s)dW_s\right)dP,
\end{equation}
under which
\begin{equation}\label{x-opt}
dx_t=f(t,x, P^*\circ x_t^{-1},u^*_t)dt+\sigma(t,x)dW^{u^*}_t,\quad
x_0=\xi\in \R^d.
\end{equation}
\\

\noindent We want to find such an optimal control and characterize the optimal cost functional $J(u^*)$.\\

\noindent For $(t,w,\mu,z,u)\in
[0,T]\times\Om\times\mathcal{P}_2(\R^d)\times\R^d\times U$ we
introduce the Hamiltonian associated with the optimal control
problem (\ref{SDE-u}) and (\ref{J-u})
\begin{equation}\label{ham-u}
H(t,w,\mu,z,u):=h(t,w,\mu,u)+z\cdot\sigma^{-1}(t,w)f(t,w,\mu,u).
\end{equation}
The function $H$ satisfies the following properties.
\begin{lemma} Assume that (A1),(A2), (B1) and (B2) hold. Then, the function $H$ satisfies
\begin{eqnarray}\label{H-u-lip}
|H(t,w,\mu,p,u)-H(t,w,\nu,p,v)|\le
C(1+|w|^{\alpha}_t))(1+|p|)(d(\mu,\nu)+\delta(u,v)).
\end{eqnarray}
Assume further that (B3) holds. Then $H$ satisfies the (stochastic)
Lipschitz condition
\begin{equation}\label{H-u-p}\begin{array}{lll}
|H(t,w,\mu,z,u)-H(t,w,\mu,z^{\prime},u)| \le
C(1+|w|^{1+\alpha}_t)|z-z^{\prime}|).
\end{array}
\end{equation}
\end{lemma}
\begin{proof} Inequality (\ref{H-u-lip}) is a consequence of (A2) and (B2). Assume further that (B3) is satisfied.
Then (\ref{H-u-p}) is also satisfied since $f$ and $\sigma^{-1}$ are
of polynomial growth in $w$.
\end{proof}

\noindent Next, we show that the payoff functional $J(u),\,u\in\U$,
can be expressed by means of solutions of  a linear BSDE.
\begin{proposition}\label{u-bsde} Assume that (A1),(A2), (B1), (B2), (B3) and (B4) are satisfied. Then, for every $u\in\U$, there exists a unique $\Ff$-progressively measurable process $(Y^u,Z^u)\in {\cs}^2_T\times {\ch}^2_T$ such that
\begin{equation}\label{u-yz-bsde}\left\{\begin{array}{ll}
-dY^u_t=H(t,x_.,P^u\circ x_t^{-1},Z^u_t,u_t) dt-Z^u_tdW_t,\quad 0\le t<T,\\
Y^u_T=g(x_T,P^u\circ x_T^{-1}).
\end{array}
\right.
\end{equation}
Moreover, $Y^u_0=J(u)$.
\end{proposition}
\begin{proof} The mapping $p\mapsto H(t,x_.,P^u\circ x_t^{-1},p,u_t)$ satisfies
(\ref{H-u-p}) and $H(t,x_.,P^u\circ x_t^{-1},0,u_t)=h(t,x_.,u_t)$
and $g(x_T,P^u\circ x_T^{-1})$ are bounded, then by Theorem I-3 in
\cite{Ham-Lepl95}, the BSDE \eqref{u-yz-bsde} has a unique solution.
\medskip

It remains to show that $Y^u_0=J(u)$. Indeed, in terms of the $(\Ff,
P^u)$-Brownian motion
$$
 W_t^u:=W_t-\int_0^t \sigma^{-1}(s,x_.)f(s,x_.,P^u\circ x_s^{-1},u_s)ds, \quad 0\le t\le T,
$$
the process $(Y^u,Z^u)$ satisfies
$$
Y^u_t=g(x_T,P^u\circ x_T^{-1})+\int_t^T h(s,x_.,P^u\circ
x_s^{-1},u_s) ds-\int_t^T Z^u_sdW^u_s,\quad 0\le t\le T.
$$
Therefore,
$$
Y^u_t=E^u\left[\int_t^T h(s,x_.,P^u\circ
x_s^{-1},u_s)ds+g(x_T,P^u\circ x_T^{-1})\big|\F_t\right]\quad
P^u\mbox{-a.s.}
$$
In particular, since $\F_0$ contains only the $P$-null sets of
$\Omega$ and, $P^u$ and $P$ are equivalent, then
$$
Y^u_0=E^u\left[\int_0^T h(s,x_.,P^u\circ
x_s^{-1},u_s)dt+g(x_T,P^u\circ x_T^{-1})\right]=J(u).
$$
\end{proof}

\subsection{Existence of optimal controls}
In the remaining part of this section we want to find $u^*\in\U$
such that $u^*=\arg\min_{u\in\U}J(u)$. A way to find such an optimal
control is to proceed as in Proposition \ref{u-bsde} and introduce a
BSDE whose solution $Y^*$ satisfies
$Y^*_0=\inf_{u\in\U}J(u)=Y^{u^*}_0$. By the comparison theorem for
BSDEs, the problem can be reduced to  minimizing the corresponding
Hamiltonian and the terminal value $g$ w.r.t. the control $u$.
Since  in the Hamiltonian $H(t,x_.,P^u\circ x_t^{-1},z,u_t)$ the marginal law $P^u\circ x_t^{-1}$ of $x_t$ under $P^u$ depends on the whole path of $u$ over $[0,t]$ and not only on $u_t$, we should minimize $H$ w.r.t. the whole set $\U$ of admissible stochastic controls. Therefore, we should take the essential infimum of the Hamiltonian over $\U$, instead of the minimum over $U$. Thus, for the associated BSDE to make sense, we should show that it exists and is progressively measurable. This is shown in the next proposition. \\

Let $\mathbb{L}$ denote the $\sigma$-algebra of progressively
measurable sets on $[0,T]\times\Omega$. For $(t,x_.,z,u)\in
[0,T]\times \C\times \R^d\times \U$, set
\begin{equation}\label{def-H}
H(t,x_.,z,u):=H(t,x_.,P^u\circ x_t^{-1},z,u_t).
\end{equation}
Note that since $H$ is linear in $z$ and a progressively measurable
process, it is an $\mathbb{L}\times B(\R^d)$-random variable.

Next we have:
\begin{proposition}\label{ess-inf} For any $z\in \R^d$, there exists an $\mathbb{L}$-measurable process $H^*(\cdot,\cdot,z)$ such that,
\begin{equation}\label{u-opt-1}
H^*(t,x_.,z)=\mathrm{ess}\inf_{u\in \U}H(t,x_.,z,u),\quad dP \times
dt \mbox{-a.s.}
\end{equation}
Moreover, $H^*$ is stochastic Lipschitz continuous in $z$, i.e., for
every $z,z^{\prime} \in\R^d$,
\begin{equation}\label{H*-lipwaw}
|H^*(t,x_.,z)-H^*(t,x_.,z^{\prime})|\le
C(1+|x|^{1+\alpha}_t)|z-z^{\prime}|.
\end{equation}
\end{proposition}

\begin{proof}
 For $n\ge 0$ let $z_n\in \mathbb{Q}^d$, the $d$-cube of rational numbers. Then, since $(t,\omega)\mapsto H(t,\omega,z_n,u)$ is $\mathbb{L}$-measurable, its essential infimum w.r.t. $u\in\U$ is well defined i.e. there exists a $\mathbb{L}$-measurable r.v. $H^n$ such that
\begin{equation}\label{H-n}
H^n(t,x_., z_n)=\underset{u\in \U}{\mathrm{ess}\inf\,}
H(t,x_.,z_n,u).
\end{equation}
Moreover, there exists a countable set $\mathcal{J}_n$ of $\U$ such
that
$$
H^n(t,x_.,z_n)=\underset{u\in \mathcal{J}_n}{\inf\,} H(t,x_.,z_n,u),
\quad dP \times dt\mbox{-a.e.}
$$ Finally note that the process  $(t,\omega)\mapsto \underset{u\in \mathcal{J}_n}{\inf\,} H(t,\omega,z_n,u)$ is $\mathbb{L}$-measurable.

Next,  set $N=\bigcup_{n\ge 0} N_n$, where
$$
N_n:=\{(t,\omega):\,\,H^n(t,\omega, z_n)\neq\underset{u\in
\mathcal{J}_n}{\inf\,} H(t,\omega,z_n,u)\}.
$$
Then obviously $dP\otimes dt(N)=0$.

We now define $H^*$ as follows : For $(t,\omega)\in N$, $H^*\equiv
0$ and for $(t,\omega)\in N^c$ (the complement of $N$) we set:
\begin{equation}\label{ess-inf-J}
H^*(t,x_.,z)=\left\{\begin{array}{ll} \underset{u\in
\mathcal{J}_n}{\inf\,} H(t,x_.,z_n,u) & \text{if }\,\,
z=z_n\in\mathbb{Q}^d, \\ \underset{z_n\in \mathbb{Q}^d\to
z}{\lim\,\,} \inf_{u\in \mathcal{J}_n} H(t,x_.,z_n,u) &
\text{otherwise. }
\end{array}
\right.
\end{equation}
The last limit exists due to the fact that, for $n\neq m$, we have
$$
\begin{array}{ll}
|\underset{u\in \mathcal{J}_n}{\inf\,} H(t,x_.,z_n,u) -\underset{u\in \mathcal{J}_m}{\inf\,} H(t,x_.,z_m,u)|\\
\quad =|H^n(t,x_.,z_n)-H^m(t,x_.,z_m)|
=|\underset{u\in \U}{\mathrm{ess}\inf\,} H(t,x_.,z_n,u)-\underset{u\in \U}{\mathrm{ess}\inf\,} H(t,x_.,z_m,u)|\\
 \quad \le \underset{u\in \U}{\mathrm{ess}\inf\,}|\sigma^{-1}(t,x_.)b(t,x_.,P^u\circ x_t^{-1},u_t)||z_n-z_m|\\\quad \le C(1+|x|_t^{\alpha+1})|z_n-z_m|.
\end{array}
$$
Furthermore, the last inequality implies that the limit does not
depend on the sequence $(z_n)_{n\geq 0}$ of
$\mathbb{Q}^d$ which converges to $z$. Finally note that $H^*(t,x_.,z)$ is $\mathbb{L}\otimes B(\R^d)$-measurable and is Lipschitz-continuous in $z$ with the stochastic Lipschitz constant $C(1+|x|_t^{\alpha+1})$.\\

\noindent  It remains to show that, for every $z\in\R^d$,
\begin{equation}\label{H-*}
H^*(t,x_.,z)=\underset{u\in \U}{\mathrm{ess}\inf\,}
H(t,x_.,z,u),\quad dP \times dt\mbox{-a.e.}
\end{equation}
If $z\in\mathbb{Q}^d$, the equality follows from the definitions
\eqref{H-n} and \eqref{ess-inf-J}. Assume $z\notin\mathbb{Q}^d$ and
let $z_n\in\mathbb{Q}^d$ such that $z_n\to z$. Then
\begin{equation}\label{H-*}
H^*(t,x_.,z_n)=\underset{u\in \U}{\mathrm{ess}\inf\,}
H(t,x_.,z_n,u),\quad dP \times dt\mbox{-a.e.}
\end{equation}
But, $H^*(t,x_.,z_n)=\underset{u\in \mathcal{J}_n}{\inf\,}
H(t,x_.,z_n,u)\rw_n H^*(t,x_.,z)$ and \\$\underset{u\in
\U}{\mathrm{ess}\inf\,} H(t,x_.,z_n,u)\rw_n\underset{u\in
\U}{\mathrm{ess}\inf\,} H(t,x_.,z,u)$ which finishes the proof.
\end{proof}

Consider further the $\F_T$-measurable random variable
\begin{equation}\label {u-opt-2}
g^*(x_.):=\mathrm{ess}\inf_{u\in \U} g(x_T,P^u\circ x_T^{-1})
\end{equation}
and let $(Y^*,Z^*)\in {\cs}^2_T\times {\ch}^2_T$ be the solution of
the following BSDE
\begin{equation}\label{*-bsde}
Y^*_t=g^*(x_.)+\int_t^T H^*(s,x_.,Z_s^*)ds-\int_t^TZ^*_sdW_s,\,\,
t\leq T.
\end{equation}

The existence of the pair $(Y^*,Z^*)$ follows from the boundedness
of $g^*$ and $h$, the measurability of $H^*$ and (\ref{H*-lipwaw})
(see \cite{Ham-Lepl95} for more details).

\medskip
The next proposition displays a comparison result between the
solutions $Y^*$  and $Y^u,\,u\in\U$ of the BSDEs \eqref{*-bsde} and
\eqref{u-yz-bsde}, respectively.
 \begin{proposition}[Comparison] For every $t\in[0,T]$, we have
 \begin{equation}\label{y*-y-u}
 Y^*_t\le Y^u_t,\quad P\text{-a.s.},\quad u\in\U.
 \end{equation}
\end{proposition}

\begin{proof} For any $t\leq T$, we have:
$$
\begin{array}{lll}
Y^*_t-Y^u_t=g^*(x_.)-g(x_T,P^u\circ x_T^{-1})-\int_t^T
(Z^*_s-Z^u_s)dW_s\\ \qquad\qquad\qquad +\int_t^T
\{H^*(s,x_.,Z_s^*)-H(s,x_.,Z^*_s,u)\} ds\\ \qquad\qquad\qquad
+\int_t^T \{H(s,x_.,Z^*_s,u) -H(s,x_.,Z^u_s,u)\} ds.
\end{array}
$$
Since, $g^*(x_.)-g(x_T,P^u\circ x_T^{-1})\le 0$ and
$H^*(s,x_.,Z_s^*)-H(t,x_.,Z^*_s,u)\le 0$, then, performing a change
of probability measure and taking conditional expectation w.r.t.
$\F_t$, we obtain  $Y^*_t\le Y^u_t,\,\, P\text{-a.s.},\,\, \forall
u\in\U$.
\end{proof}

\begin{proposition}[$\e$-optimality] Assume that for any $\e>0$ there exists $u^{\e}\in \U$ such that $P$-a.s.,
\begin{equation}\label{e-*-H-g}
\left\{ \begin{array}{ll} H^*(t,x_.,Z^*_t)\ge
H(t,x_.,Z^*_t,u^{\e})-\e, \quad 0\le t<T, \\ g^*(x_.)\ge
g(x_T,P^{u^{\e}}\circ x_T^{-1})-\e.
\end{array}
\right.
\end{equation}
Then,
\begin{equation}\label{*-e-opt}
Y^*_t=\underset{u\in\U}{\mathrm{ess}\inf\,} Y^u_t,\quad 0\le t\le T.
\end{equation}
\end{proposition}
\begin{proof} Let $(Y^{\e},Z^{\e})\in {\cs}^2_T\times {\ch}^2_T$ be the solution the following BSDE
$$
Y^{\e}_t=g(x_T,P^{u^{\e}}\circ x_T^{-1})+\int_t^T
H(s,x_.,Z^{\e}_s,u^{\e})ds-\int_t^T Z_s^{\e}dW_s.
$$ Once more the existence of $(Y^{\e},Z^{\e})$ follows from (\cite{Ham-Lepl95}, Theorem I.3).
We then have
$$
\begin{array}{lll}
Y^*_t-Y^{\e}_t=g^*(x_.)-g(x_T,P^{u^{\e}}\circ x_T^{-1})-\int_t^T
(Z^*_s-Z^{\e}_s)dW_s\\ \qquad\qquad +\int_t^T
\{H^*(s,x_.,Z^*)-H(s,x_.,Z^*_s,u^{\e})\} ds\\ \qquad\qquad +\int_t^T
\{H(s,x_.,Z^*_s,u^{\e}) -H(s,x_.,Z^{\e}_s,u^{\e})\} ds.
\end{array}
$$
Since $g^*(x_.)-g(x_T,P^{u^{\e}}\circ x_T^{-1})\ge -\e $ and
$H^*(s,x_.,Z^*)-H(t,x_.,Z^*_s,u^{\e})\ge -\e$, then, once more,
performing a change of probability measure and taking conditional
expectation w.r.t. $\F_t$, we obtain   $Y^*_t\ge Y^{u^{\e}}_t
-\e(T+1)$. This entails that, in view of (\ref{y*-y-u}), for every
$0\le t\le T$, $Y^*_t=\underset{u\in\U}{\mathrm{ess}\inf\,}Y^u_t$ .
\end{proof}

In next theorem, we characterize the set of optimal controls
associated with  (\ref{opt-J}) under the dynamics (\ref{SDE-u}).
\begin{theorem}[Existence of optimal control]\label{opt-BSDE} If there exists $u^*\in\U$ such that
\begin{equation}\label{*-H-g-opt}
\left\{ \begin{array}{ll}  H^*(t,x_.,Z^*_t)=H(t,x_.,P^{u^*}\circ
x^{-1}_t,Z^*_t,u^*),\quad dP\times dt\mbox{-a.e.},\quad 0\le t< T,
\\
g^*(x_.)=g(x_T,P^{u^*}\circ x^{-1}_T), \quad dP\mbox{-a.s.}
\end{array}
\right.
\end{equation}
Then, \begin{equation}\label{Y-opt}
Y^*_t=Y^{u^*}_t=\mathrm{ess}\inf_{u\in\U} Y^u_t,\quad 0\le t\le T.
\end{equation}
In particular, $Y_0^*=\inf_{u\in\U}J(u)=J(u^*)$.
\end{theorem}

\begin{proof} Under (\ref{*-H-g-opt}), for any $t\leq T$ we have
$$
\begin{array}{lll}
Y^*_t-Y^{u^*}_t=\int_t^T (Z^*_s-Z^{u^*}_s)dW_s+\int_t^T
\{H(s,x_.,P^{u^*}\circ x^{-1}_t,Z^*_s,u^{*}) -H(s,x_.,P^{u^*}\circ
x^{-1}_t,Z^{u^*}_s,u^{*})\} ds\\ \qquad\qquad = \int_t^T
(Z^*_s-Z^{u^*}_s)dW_s+\int_t^T
(Z^*_s-Z^{u^*}_s)\sigma^{-1}(s,x.)f(s,x_.,P^{u^*}\circ
x^{-1}_s,u^{*})ds .
\end{array}
$$
Making now a change of probability and taking expectation leads to
$\tilde E[Y^*_t-Y^{u^*}_t]=0$, $\forall t\leq T$ where $\tilde E$ is
the expectation under the new probability $\tilde P$ which is
equivalent to $P$. As $Y^*_t\leq Y^{u^*}_t$, $ P$-a.s. and then
$\tilde P$-a.s., we obtain, in taking into account of
(\ref{y*-y-u}), $Y^*=Y^{u^*}$ which means, once more by
(\ref{y*-y-u}), that $u^*$ is an optimal strategy.\end{proof}

\begin{remark}
As is the case for any optimality criteria for systems, obviously
checking the sufficient condition \eqref{*-H-g-opt} is quite hard
simply because there are no general conditions which guarantee
existence of essential minima for systems.  One should rather solve
the problem in particular cases. In the special case where the
marginal law $P^u\circ x^{-1}_t$ {\it only depends} on $(u_t, x_.)$
at each time $t\in [0,T]$,  we may minimize $H$ and $g$ over the
action set $U$, instead of using the essential infimum, and use
Bene\v{s} selection theorem \cite{Benes} to find two measurable
functions $u_1^*$ from $[0,T)\times\Om\times\R^d$ into $U$ and
$u_2^*$ from $\R^d$ into $U$ such that
\begin{equation}
\label{definf}H^*(t,x_.,z):=\inf_{u\in U}H(t,x_.,P^u\circ
x_t^{-1},z,u)= H(t,x,P^{u_1^*}\circ x_t^{-1},z,u_1^*(t,x,z))
\end{equation}  and
\begin{equation}\label {u-opt-2}
g^*(x_.):=\inf_{u\in U} g(x_T,P^u\circ
x_T^{-1})=g(x_T,P^{u_2^*}\circ x_T^{-1}).
\end{equation}
Combining (\ref{definf}) and (\ref{u-opt-2}), it is easily seen that
the progressively measurable function $u^*$ defined by
\begin{equation}\label{u-opt-hat1}
\widehat u(t,x_.,z):=\left\{\begin{array}{ll} u_1^*(t,x_.,z), \quad
t<T,\\ u_2^*(x_T),\quad t=T,
\end{array}
\right.
\end{equation}
satisfies
\begin{equation}\label{opt}
H^*(t,x_.,z)=H(t,x_.,P^{\widehat u}\circ x_t^{-1},z,\widehat u)
\quad  \text{and} \quad   g^*(x_.)=g(x_T,P^{\widehat u}\circ
x_T^{-1}).
\end{equation}
\end{remark}

\subsection{Existence of nearly-optimal controls}

As noted above, the sufficient condition \eqref{*-H-g-opt} is quite
hard to verify in concrete situations, which makes Theorem
\eqref{opt-BSDE} less useful for showing existence of optimal
controls. Nevertheless,
near-optimal controls enjoy many useful and desirable properties that optimal controls do not have.
In fact, thanks to Ekeland's variational principle \cite{Ekeland}, that we will use below, under
very mild conditions on the control set $\U$ and the payoff functional $J$, near-optimal
controls always exist while optimal controls may not exist or are difficult to establish.
Moreover, there are many candidates for near-optimal controls which makes it possible
to select among them appropriate ones that are easier to implement and handle both analytically
and numerically. \\

We introduce the Ekeland metric $d_E$ on the space
$\U$ of admissible controls defined as follows. For $u, v\in \U$,
\begin{equation}\label{E-distance}
d_E(u,v):=\widehat P\{(\omega,t)\in\Omega\times[0,T],\,\,
\delta(u_t(\omega),v_t(\omega))>0\},
\end{equation}
where $\widehat P$ is the product measure of $P$ and the Lebesgue
measure on $[0,T]$.\\

In our proof of existence of near-optimal controls, we need $L^p$-boundedness of the Girsanov density $L^u$ for some $p>1$, which, accoding to Theorem 2.2 in \cite{ugh},  is achieved  under the following
assumption on $\sigma$ which will replace (A2)-(b),(c).\\

\noindent {\bf Assumption} (A6): $\sigma (t,x_.)$ and
$\sigma^{-1}(t,x_.)$ are bounded.\\

We have
\begin{lemma}\label{ekeland}
\begin{itemize}
\item[$(i)$] $d_E$ is a distance. Moreover, $(\U,d_E)$  is a complete metric space.
 \item[$(ii)$]  Let $(u^n)_n$ and $u$  be in $\U$. If $d_E(u^n,u)\to 0$ then $\E[\int_0^T\delta^2(u^n_t,u_t)dt]\to 0$.
\end{itemize}
\end{lemma}

\begin{proof}
For a proof of $(i)$, see \cite{Elliott}. The proof of completeness of $(\U,d_E)$ needs only completeness of the metric space $(U,\delta)$.\\
$(ii)$ Let $(u^n)_n$ and $u$ be in $\U$. Then, by definition of the
distance $d_E$,  since $d_E(u^n,u)\to 0$ then $\delta(u^n_t,u_t)$
converges to 0, $dP\times dt$-a.e. Now, since the set $U$ is
compact, the sequence $\delta(u^n,u)$ is  bounded. Thus, by
dominated convergence, we have $\E[\int_0^T\delta^2(u^n_t,u_t)dt]\to
0$.
\end{proof}

\begin{proposition}\label{D-conv}
Assume (A1), (A2)-(a),(A6) and (B1)-(B4). Let $(u^n)_n$ and $u$ be
in $\U$. If $d_E(u^n,u)\to 0$ then $D^2_T(P^{u^n},P^u)\to 0$.
Moreover, for every $t\in[0,T]$, $L^{u^n}_t$ converges to $L^{u}_t$
in $L^1(P)$.
\end{proposition}
\begin{proof}
In view of  Lemma \eqref{ekeland}, we have $\E[\int_0^T
\delta^2(u_t, u^n_t)dt] \to 0$. Therefore the sequence $(\int_0^T
\delta^2(u_t, u^n_t)dt)_{n\ge 0}$ converges in probability w.r.t $P$
to 0 and by compacity of $U$ it is bounded. On the other hand since
$L^u_T$ is integrable then the sequence $(L^u_T\int_0^T
\delta^2(u_t, u^n_t)dt)_{n\ge 0}$ converges also in probability wrt
to $P$ to 0. Next by the uniform boundedness of \\$(\int_0^T
\delta^2(u_t, u^n_t)dt)_{n\ge 0}$, the sequence $(L^u_T\int_0^T
\delta^2(u_t, u^n_t)dt)_{n\ge 0}$ is uniformly integrable. Finally
as we have
$$
\E^u[\int_0^T \delta^2(u_t, u^n_t)dt]=\E[L^u_T\int_0^T \delta^2(u_t,
u^n_t)dt]
$$ then
$$
\E^u[\int_0^T \delta^2(u_t, u^n_t)dt]\rightarrow_n 0.
$$
Now, to
conclude it is enough to use the inequality \eqref{TV-uv-1}.

To prove the last statement, set $M^u_t:=\int_0^t f(s,x_.,P^u\circ
x_s^{-1},u_s)dW_s$. In view of (B2), we have
$$
\begin{array}{lll}
\E[|M^{u_n}_t-M^u_t|^2]=\E[\int_0^t |f(s,x_.,P^{u^n}\circ
x_s^{-1},u^n_s)-f(s,x_.,P^u\circ x_s^{-1},u_s)|^2ds]\\ \qquad\qquad
\qquad \quad \,\,\le C(D_t(P^{u_n},P^u)+E[\int_0^T \delta^2(u^n_t,
u_t)dt],
\end{array}
$$
which converge to zero as $n\to +\infty$. \\
Furthermore,  setting $f(t,x_.,u):=f(t,x_.,P^{u}\circ
x_t^{-1},u_t)$, we have
$$
\begin{array}{lll}
\E[|\langle M^{u^n}\rangle_t-\langle M^u\rangle_t|]\le \E[\int_0^t
|f(s,x_.,u^n)-f(s,x_.,u)|(|f(s,x_.,u^n)|+|f(s,x_.,u)|)ds] \\
\qquad\qquad  \quad \le (\E[\int_0^t
|f(s,x_.,u^n)-f(s,x_.,u)|^2])^{1/2}(\E[\int_0^t
(|f(s,x_.,u^n)|+|f(s,x_.,u)|)^2ds])^{1/2} \\ \qquad\qquad  \quad \le
C(E[\int_0^t |f(s,x_.,u^n)-f(s,x_.,u)|^2])^{1/2}\E[\int_0^t
(1+|x|_s^2)ds])^{1/2}
\end{array}
$$
which converges to zero as $n\to +\infty$. Therefore, $L^{u^n}_t$
converges to $L^u_t$ in probability w.r.t. $P$. But, by Theorem 2.2 in
\cite{ugh}, under (A6) and (B3), $(L^{u^n}_t)_n$ is uniformly integrable.
Thus, $L^{u^n}_t$ converges to $L^u_t$  in $L^1(P)$ when $n\to
+\infty$.
\end{proof}
\begin{proposition}\label{near-opt}
For any $\e>0$, there exists a control $u^{\e}\in\U$ such that

\begin{equation}\label{e-opt}
J(u^{\e})\le \inf_{u\in\U} J(u)+\e.
\end{equation}
$u^{\e}$ is called  near or $\e$-optimal for the payoff functional
$J$.
\end{proposition}
\begin{proof} The result follows from Ekeland's variational principle, provided that we prove that the payoff function $J$, as a mapping from the complete metric space $(\U,d_E)$ to $\R$, is lower bounded and lower-semicontinuous. Since $f$ and $g$ are assumed uniformly bounded, $J$ is obviously bounded. We now show continuity of $J$: $J(u^n)$ converges to $J(u)$ when  $d_E(u^n,u)\to 0$. \\
Integrating by parts, we obtain
$$
J(u)=E[\int_0^T L^u_th(t,x_.,P^u\circ x_t^{-1},u_t)dt+L^u_Tg(x_T,
P^u\circ x_T^{-1})].
$$
Using the inequality
$$
|L^{u^n}_th(t,x_.,u^n)-L^{u}_th(t,x_.,u)|\le
|L^{u^n}_t-L^{u}_t|h(t,x_.,u)|+L^{u}_t|h(t,x_.,u^n)-h(t,x_.,u)|
$$
and (B3) together with the boundedness of $h$, by Proposition
\eqref{D-conv}, $E[\int_0^T L^{u^n}_th(t,x_.,P^{u^n}\circ
x_t^{-1},u^n_t)dt]$ converges to $E[\int_0^T L^u_th(t,x_.,P^u\circ
x_t^{-1},u_t)dt]$ as $d_E(u^n,u)\to 0$. A similar argument yields
convergence of $E[L^{u^n}_Tg(x_T, P^{u^n}\circ x_T^{-1})]$ to
$E[L^u_Tg(x_T, P^u\circ x_T^{-1})]$ when $d_E(u^n,u)\to 0$.
\end{proof}


\section{The zero-sum game problem}\label{zero-sum}
In this section we consider a two-players zero-sum game.  Let $\U$ (resp. $\V$) be the set of admissible $U$-valued (resp. $V$-valued) control strategies for the first (resp. second) player, where $(U,\delta_1)$ and $(V,\delta_2)$ are compact metric spaces.\\

\noindent  For $(u,v),(\bar u,\bar v)\in U\times V$, we set
\begin{equation}\label{delta-u-v}
\delta((u,v),(\bar u,\bar v)):=\delta_1(u,\bar u)+\delta_2(v,\bar
v).
\end{equation}
The distance $\delta$ defines a metric on the compact space $U\times V$. \\

\noindent Let $f$ and $h$ be two measurable functions from $[0,T]\times\Om\times\mathcal{P}_2(\R^d)\times U\times V$ into $\R^d$ and $\R$, respectively,  and $g$ be a measurable function from $\R^d\times\mathcal{P}_2(\R^d)$ into $\R$ such that \\

\begin{itemize}
 \item[(C1)] For any $(u,v)\in\U\times\V$ and $Q\in \mathcal{P}_2(\Om)$, the processes $(f(t, x_.,Q\circ x_t^{-1},u_t,v_t))_t$ and $(h(t, x_.,Q\circ x_t^{-1},u_t,v_t))_t$ are progressively measurable.  Moreover, $g(x_T,Q\circ x_T^{-1})$ is $\mathcal{F}_T$-measurable.

 \item[(C2)] For every $t\in[0,T]$, $w\in\Om$, $(u,v),(\bar u,\bar v) \in U\times V$ and $\mu, \nu \in\mathcal{P}(\R^d)$,
  $$
  |\phi(t,w,\mu, u,v)-\phi(t,w,\nu,\bar u,\bar v)|\le C(d(\mu,\nu)+\delta((u,v),(\bar u,\bar v)),
  $$
  for $\phi\in\{f,h,g\}$.
 \item[(C3)]  For every $t\in[0,T]$, $w\in\Om,\,\mu\in\mathcal{P}(\R^d)$ and  $(u,v)\in\U\times\V$,
$$
 |f(t,w,\mu,u,v)|\le C(1+|w|_t).
 $$
\item[(C4)] $h$ and $g$ are uniformly  bounded.
 \end{itemize}
 \medskip

\noindent For $(u,v)\in\U\times \V$, let $P^{u,v}$ be the
probability measure on $(\Om,\F)$ defined by
 \begin{equation}\label{P-u-v}
 dP^{u,v}:=L_T^{u,v}dP,
 \end{equation}
where
\begin{equation}\label{P-u-v-density}
 L_t^{u,v}:=\mathcal{E}_t\left(\int_0^{\cdot} \sigma^{-1}(s,x_.)f(s,x_.,P^{u,v}\circ x_s^{-1},u_s,v_s)dW_s\right),\quad 0\le t\le T.
 \end{equation} The proof of existence of $P^{u,v}$ follows the same lines as the one of $P^u$ defined in (\ref{P-u})-(\ref{P-u-density}).
Hence, by Girsanov's theorem, the process $(W^{u,v}_t,\,\, 0\le t\le
T)$ defined by
 $$
 W_t^{u,v}:=W_t-\int_0^t \sigma^{-1}(s,x_.)f(s,x_.,P^{u,v}\circ x_s^{-1},u_s,v_s)ds, \quad 0\le t\le T,
 $$
 is an $(\Ff, P^{u,v})$-Brownian motion.  Moreover, under $P^{u,v}$,
 \begin{equation}\label{SDE-u-v}
 dx_t=f(t,x_.,P^{u,v}\circ x_t^{-1},u_t,v_t)dt+\sigma(t,x_.)dW^{u,v}_t,\quad x_0=\xi\in \R^d.
 \end{equation}
Let $E^{u,v}$  denote the expectation w.r.t. $P^{u,v}$.\\
\noindent The payoff functional $J(u,v),\,(u,v)\in\U\times\V$,
associated with the controlled SDE (\ref{SDE-u-v}) is
\begin{equation}\label{J-u-v}
J(u,v):=E^{u,v}\left[\int_0^T h(t,x_.,P^{u,v}\circ
x_t^{-1},u_t,v_t)dt+ g(x_T,P^{u,v}\circ x^{-1}_T)\right].
\end{equation}
  \\
The zero-sum game we consider is between two players, where the
first player (with control $u$) wants to minimize the payoff
(\ref{J-u-v}), while  the second player (with control $v$) wants to
maximize it.  The zero-sum game boils down to showing existence of a
saddle-point for the game i.e.  to show existence of a pair $(u^*,
v^*)$ of strategies such that
\begin{equation}\label{J-u-v-hat}
J(u^*, v) \le J(u^*, v^*)\le J(u,v^*)
\end{equation}
for each $(u, v)\in\U\times\V$.\\
 The corresponding dynamics is given by the probability measure $P^*$ on $(\Om,\F)$ defined by
\begin{equation}\label{opt-P}
dP^*=\mathcal{E}_T\left(\int_0^{\cdot} \sigma^{-1}(s,x_.)f(s,x_.,
P^*\circ x_s^{-1},u^*_s, v^*_s)dW_s\right)dP
\end{equation}
under which
\begin{equation}\label{x-opt}
dx_t=f(t,x, P^*\circ x_t^{-1},u^*_t,v^*_t)dt+\sigma(t,x)dW^{u^*,
v^*}_t,\quad x_0=\xi\in \R^d.
\end{equation}

\medskip

For $(u,v)\in\U\times\V$ and $z\in\R^d$, we introduce the
Hamiltonian associated with the game  (\ref{SDE-u-v})-(\ref{J-u-v}):
\begin{equation}\label{ham-u-v}\begin{array}{lll}
H(t,x_.,z,u,v):=z\cdot\sigma^{-1}(t,x_.)f(t,x_.,P^{u,v}\circ x_t^{-1},u_t,v_t)\\
\qquad\qquad\qquad\qquad \qquad\qquad\qquad \qquad\qquad+
h(t,x_.,P^{u,v}\circ x_t^{-1},u_t,v_t).
\end{array}
\end{equation}
Next, set
\begin{itemize}
\item $\underline{H}(t,x_.,z):=\underset{v\in\V}{\mathrm{ess}\sup}\, \underset{u\in\U}{\mathrm{ess}\inf}\, H(t,x_.,z,u,v),$
\item $\overline{H}(t,x_.,z):=\underset{u\in\U}{\mathrm{ess}\inf}\, \underset{v\in\V}{\mathrm{ess}\sup}\, H(t,x_.,z,u,v),$
\item $\underline{g}(x_.):=\underset{v\in\V}{\mathrm{ess}\sup}\, \underset{u\in\U}{\mathrm{ess}\inf}\, g(x_T, P^{u,v}\circ x_T^{-1})$,
\item $\overline{g}(x_.):=\underset{u\in\U}{\mathrm{ess}\inf}\, \underset{v\in\V}{\mathrm{ess}\sup}\, g(x_T, P^{u,v}\circ x_T^{-1})$.
\end{itemize}
As in Proposition \ref{ess-inf}, $\underline H$, $\overline H$, $\underline g$ and $\overline g$ exist. On the other hand following a similar proof as the one leading to \eqref{H*-lipwaw}, $\underline{H}(t,x_.,z)$ and $\overline{H}(t,x_.,z)$ are stochastic Lipschitz continuous in $z$ with the Lipschitz constant $C(1+|x|^{1+\alpha}_t)$. \\

\noindent Let $(\underline{Y},\underline{Z})$ be the solution of the
BSDE associated with $(\underline{H}, \underline{g})$ and
$(\overline{Y},\overline{Z})$ the solution of the BSDE associated
with $(\overline{H}, \overline{g})$.

\begin{definition}[Isaacs' condition]
We say that the Isaacs' condition holds for the game if
$$
\left\{\begin{array}{lll}
\underline{H}(t,x_.,z)=\overline{H}(t,x_.,z),\quad  z\in\R^d,\,\, 0\le t\le T, \\
\underline{g}(x_.)=\overline{g}(x_.),
\end{array}
\right.
$$
\end{definition}

\noindent Applying the comparison theorem for BSDEs and then
uniqueness of the solution, we obtain the following

\begin{proposition}\label{game-comparison} For every $t\in[0,T]$, it holds that $\underline{Y}_t\le \overline{Y}_t$, $\,P$-a.s. Moreover, if the Issac's condition holds, then
\begin{equation}\label{nash}
\underline{Y}_t=\overline{Y}_t:=Y_t,\quad P\mbox{-a.s.},\quad 0\le
t\le T.
\end{equation}
\end{proposition}

In the next theorem, we formulate conditions for which the zero-sum
game has a value.  For $(u,v)\in\U\times \V$, let
$(Y^{u,v},Z^{u,v})\in {\cs}^2_T\times {\ch}^2_T$ be the solution of
the BSDE
\begin{equation}\label{u-v-yz-bsde}\left\{\begin{array}{ll}
-dY^{u,v}_t=H(t,x_.,Z^{u,v}_t,u,v) dt-Z^{u,v}_tdW_t,\quad 0\le t<T,\\
Y^{u,v}_T=g(x_T,P^{u,v}\circ x_T^{-1}),
\end{array}
\right.
\end{equation}

\begin{theorem}[Existence of a value of the zero-sum game]\label{value-game}
Assume that, for every $t\in[0,T]$,
\begin{equation}\label{hypothz} \underline{H}(t,x_.,\underline Z_t)=\overline{H}(t,x_.,\underline Z_t).
\end{equation}
If there exists $(u^*,v^*)\in\U\times\V$ such that, for every $0\le
t<T$,
\begin{equation}\label{sp-H}
\underline{H}(t,x_.,\underline Z_t)=\underset{u\in
\U}{\mathrm{ess}\inf}\, H(t,x_.,\underline Z_t,u,v^*)=\underset{v\in
\V}{\mathrm{ess}\sup}\, H(t,x_.,\underline Z_t,u^*,v),
\end{equation}
and
\begin{equation}\label{sp-g}
\underline{g}(x_.)=\overline{g}(x_.)=\underset{u\in
\U}{\mathrm{ess}\inf}\, g(x_T,P^{u,v^*}\circ
x_T^{-1})=\underset{v\in \V}{\mathrm{ess}\sup}\, \,
g(x_T,P^{u^*,v}\circ x_T^{-1}).
\end{equation}
Then,
\begin{equation}\label{value}
Y_t=\underset{u\in \U}{\mathrm{ess}\inf}\,\underset{v\in
\V}{\mathrm{ess}\sup}Y_t^{u,v}=\underset{v\in
\V}{\mathrm{ess}\sup}\,\underset{u\in
\U}{\mathrm{ess}\inf}\,Y_t^{u,v}, \quad 0\le t\le T.
\end{equation}
Moreover, the pair $(u^*,v^*)$is a saddle-point for the game.
\end{theorem}

\begin{proof} First note that we can replace in (\ref{hypothz}) $\underline Z$ by $\overline Z$ and the result still holds. So assume that $\underline{H}(t,x_.,\underline Z_t)=\overline{H}(t,x_.,\underline Z_t)$. Then by the uniqueness of the solution of the BSDEs associated with $(\underline H,\underline g)$ and
$(\overline H,\overline g)$ we have $(\underline{Y},\underline
Z)=\overline{Y},\overline Z)$.
\medskip

On the other hand, by (\ref{sp-H})-(\ref{sp-g})  one can easily
check that the paire $(u^*,v^*)$ satisfies a saddle-point property
for $H$ and $g$ as well, i.e.,
$$
H(t,x_.,\underline Z_t,u^*,v)\leq \underline{H}(t,x_.,\underline
Z_t)={H}(t,x_.,\underline Z_t,u^*,v^*)\leq H(t,x_.,\underline
Z_t,u,v^*), t<T
$$
and
$$
g(x_T,P^{u^*,v}\circ x_T^{-1})\leq
\underline{g}(x_.)=\overline{g}(x_.)=g(x_T,P^{u^*,v^*}\circ
x_T^{-1})\leq  g(x_T,P^{u,v^*}\circ x_T^{-1}).
$$
The previous equalities and the uniquess of the solutions of the
BSDEs imply that $\overline Y_t=\underline Y_t=Y^{u^*,v^*}_t$.

Now let $(u,v)\in\U\times \V$ and, $(\widehat Y^u,\widehat Z^u)$,
$(\widetilde Y^v,\widetilde Z^v)$ be the solutions of the following
BSDEs:
\begin{equation}\label{u-v-yz-bsde}\left\{\begin{array}{ll}
-d\widehat Y^{u}_t=\underset{v\in \V}{\mathrm{ess}\sup}\,H(t,x_.,\widehat Z^{u}_t,u,v) dt-\widehat Z^{u}_tdW_t,\quad 0\le t<T,\\
\widehat Y^{u}_T=\underset{v\in
\V}{\mathrm{ess}\sup}\,g(x_T,P^{u,v}\circ x_T^{-1}),
\end{array}
\right.
\end{equation}
\begin{equation}\label{u-v-yz-bsde}\left\{\begin{array}{ll}
-d\widetilde Y^{v}_t=\underset{u\in \U}{\mathrm{ess}\inf}\, H(t,x_.,\widetilde Z^{v}_t,u,v) dt-\widetilde Z^{v}_tdW_t,\quad 0\le t<T,\\
\widetilde Y^{v}_T=\underset{u\in
\U}{\mathrm{ess}\inf}\,g(x_T,P^{u,v}\circ x_T^{-1}).
\end{array}
\right.
\end{equation}
Then by comparison we have \begin{equation}\label{*-*} \hat
Y^{u^*}_t\geq Y^{u^*,v}_t \mbox{ and }\tilde  Y^{v^*}_t\leq
Y^{u,v^*}_t.
\end{equation} But $\hat  Y^{u^*}$ satisfies the following BSDE:
\begin{equation}\label{u-v-yz-bsde1}\left\{\begin{array}{ll}
-d\widehat Y^{u^*}_t=\underset{v\in \V}{\mathrm{ess}\sup}\,H(t,x_.,\widehat Z^{u^*}_t,u^*,v) dt-\widehat Z^{u^*}_tdW_t,\quad 0\le t<T,\\
\widehat Y^{u^*}_T=\underset{v\in
\V}{\mathrm{ess}\sup}\,g(x_T,P^{u^*,v}\circ x_T^{-1}).
\end{array}
\right.
\end{equation} Taking into account of (\ref{sp-H})-(\ref{sp-g}) and since the solution of the previous BSDE is unique, we obtain that
$$
\underline Y_t=Y^{u^*,v^*}_t=\hat Y^{u^*}_t.
$$
Moreover, (\ref{*-*}) implies that $Y^{u^*,v^*}_t\geq Y^{u^*,v}_t$
for any $v \in \V$.  But in the same way we have also $\underline
Y_t=Y^{u^*,v^*}_t=\tilde Y^{v^*}_t\leq Y^{u,v^*}_t$, $P$-a.s., for
any $u\in \U$. Therefore,
$$
Y^{u^*, v}_t\le Y^{u^*, v^*}_t\le Y^{u, v^*}_t.
$$
Thus, $(u^*,v^*)$ is a saddle-point of the game and $\underline
Y_t=Y^{u^*, v^*}_t$ is the value of the game, i.e., it satisfies
$$
Y^{u^*, v^*}_t=Y_t=\underset{u\in
\U}{\mathrm{ess}\inf}\,\underset{v\in
\V}{\mathrm{ess}\sup}Y_t^{u,v}=\underset{v\in
\V}{\mathrm{ess}\sup}\,\underset{u\in
\U}{\mathrm{ess}\inf}\,Y_t^{u,v}, \quad 0\le t\le T.
$$
\end{proof}

\noindent {\bf Final remark} Assumptions (B4) and (C4) on the
boundedness of the functions $g$ and $h$  can be substantially
weakened by using subtle arguments on existence and uniqueness of
solutions of one dimensional BSDEs which are by now well known in
the BSDEs literature.


\end{document}